\documentclass[12pt]{article}
\usepackage{amsfonts,epsf,amsmath,amssymb,graphicx,amsthm,fullpage}

\newtheorem{theorem}{\bf Theorem}[section]
\newtheorem{corollary}[theorem]{\bf Corollary}
\newtheorem{lemma}[theorem]{\bf Lemma}

\theoremstyle{definition}

\def\dern{{\rm dern}}
\def\adern{{\rm adern}}

\def\cS{{\mathcal S}}

\usepackage{amssymb}
\usepackage{pst-all}
\begin{document}

\title{Degree-associated edge-reconstruction numbers of
double-brooms}
\author
{{Meijie Ma$^{1,2}$\thanks{Corresponding author: Meijie Ma,
E-mail: mameij@mail.ustc.edu.cn}, \ Tingting
Zhou$^{2}$}\\
{\small $^1$School of Management Science and Engineering,}\\
{\small Shandong Technology and Business University, Yantai
264005,
China} \\
{\small $^2$Department of Mathematics,}\\
{\small Zhejiang
Normal University, Jinhua 321004, China }}

\date{}

\maketitle 

\baselineskip16pt

\begin{abstract}
An edge-deleted subgraph of a graph $G$ is an {\it edge-card}. A
{\it decard} consists of an edge-card and the degree of the
missing edge.  The {\it degree-associated edge-reconstruction
number} of a graph $G$, denoted $\dern(G)$, is the minimum
number of decards that suffice to reconstruct $G$.  The {\it
adversary degree-associated edge-reconstruction number}
$\adern(G)$ is the least $k$ such that every set of $k$ decards
determines $G$. We determine these two parameters for all
double-brooms. The answer is usually $1$ for $\dern(G)$,
and $2$ for $\adern(G)$ when $G$ is double-broom.
But there are exceptions in each case.

\end{abstract}



\baselineskip17.8pt

\section{Introduction}
\label{sec:intro}

The Reconstruction Conjecture of Kelly~\cite{kel1,kel2} and
Ulam~\cite{U} has been open for more than 50 years.  An
unlabeled subgraph of a graph $G$ obtained by deleting one
vertex is a {\it card}.  The multiset of cards is the {\it deck}
of $G$. The Reconstruction Conjecture asserts that every graph
with at least three vertices is uniquely determined by its deck.
Such a graph is {\it reconstructible}.

An {\it edge-card} of a graph $G$ is a subgraph obtained from
$G$ by deleting one edge. Edge-cards are unlabeled; only the
isomorphism class is known. The multiset of edge-cards is the
{\it edge-deck} of $G$. The Edge-Reconstruction Conjecture of
Harary~\cite{H} states that every graph with more than three
edges is determined by its edge-deck (the claw $K_{1,3}$ and the
disjoint union of a triangle and one vertex have the same
edge-deck).  A graph that is determined by its edge-deck is {\it
edge reconstructible}. The Edge-Reconstruction Conjecture is
simply the statement that line graphs are reconstructible.

The {\it degree of an edge} $e$, denoted by $d(e)$, is the
number of edges adjacent to $e$. A  {\it degree-associated
edge-card} or {\it decard} is a pair $(G-e,d(e))$ consisting of
an edge-card and the degree of the missing edge. The {\it
dedeck} is the multiset of decards. For an edge reconstructible
graph $G$, the {\it degree-associated edge-reconstruction number
of $G$}, denoted $\dern(G)$, is the minimum number of decards
that suffice to reconstruct $G$. The {\it  adversary
degree-associated edge-reconstruction number} $\adern(G)$ is the
least $k$ such that every set of $k$ decards determines $G$. The
definitions of degree-associated reconstruction number and
adversary degree-associated reconstruction number are analogous
to the edge setting. There are some papers about them~\cite{BW,
MSW, MSJS}.

The study of $\dern$ and $\adern$ was initiated by Monikandan
and Sundar Raj~\cite{mse}. They determined $\dern(G)$ and
$\adern(G)$ when $G$ is a regular graph, a complete bipartite
graph, a path, a wheel, or a double-star. They also proved that
$\dern(G)\leq 2$ and $\adern(G)\leq 3$ when $G$ is a complete
$3$-partite set whose part-sizes differ by at most $1$. The
authors also determined the parameters for several product
graphs~\cite{mds}.

In this paper, we determine $\dern$ and $\adern$ for all
double-brooms. The {\it double-broom} $D_{m,n,p}$ with $p\geq 2$
is the tree with $m+n+p$ vertices obtained from a $p$-vertex
path appending $m$ leaf neighbors at one end and $n$ leaf
neighbors at the other end. When $p=2$, $D_{m,n,2}$ is also
named as {\it double-star}. By symmetry, we may assume $1\leq
m\leq n$. For $D_{m,n,p}$, we distinguish three types of edges.
A {\it leaf edge} is an edge incident to a leaf, a {\it middle
edge} is an edge with degree $2$ other than a leaf edge, a {\it
hub edge} is an edge with degree more than $2$ other than a leaf
edge. A {\it leaf decard}, {\it middle decard}, or {\it hub
decard} is a decard obtained by deleting a leaf, middle, or hub
edge, respectively.

We refer to the two kinds of leaf edge in $D_{m,n,p}$ as the
{\it left leaf edge}(with degree $m$) and the  {\it right leaf
edge}(with degree $n$). In the degenerate case $m=0$, we write
$D_{0,n,p}=D_{1,n,p-1}$ for $p>2$. We use $``+"$ to denote {\it
disjoint union}.  When $m,n\ge1$, the left and right leaf
decards are copies of $(D_{m-1,n,p}+K_1, m)$ and
$(D_{m,n-1,p}+K_1,n)$ respectively. The edge-card obtained by
deleting a middle edge or hub edge is a disjoint union of two
brooms, where the  broom $B_{m,a}$ is the tree with $m+a$
vertices obtained from a path with $a$ vertices by adding $m$
leaf neighbors to one endpoint of the path. Note that $B_{m,1}$
is a star $K_{1,m}$, and $B_{m,2}$ is a star $K_{1,m+1}$. The
middle decards of $D_{m,n,p}$ are expressions of the form
$(B_{m,a}+B_{n,p-a},2)$, where $1\leq a\leq p-1$. A middle
decard occurs twice in the dedeck of $D_{m,n,p}$ only if $m=n$,
otherwise once. The number of middle decards is $p-1$ minus the
number of the hub edges.

We prove some useful lemmas in Section $2$. Our main results are
given in Section $3$.

\section{Some Lemmas} \label{sec:results}

When we prove a set of decards $\cS$ determines the graph $G$,
the main idea is finding all possible reconstructions from a
decard in $\cS$. If the dedeck of a reconstruction other than
$G$ cannot contain $\cS$ as a subset, the set $\cS$ determines
the graph $G$. For convenience, an $i$-vertex is a vertex with
degree $i$ and a $j$-edge is an edge with degree $j$. We use
$C_l$ to denote a cycle with $l$ vertices.

\begin{lemma} [\cite{MSSW}]\label{leme1}
If a graph $G$ has an edge $e$ such that $d(e)=0$ or no two
non-adjacent vertices in $G-e$ other than the endpoints of $e$
have degree-sum $d(e)$, then the decard $(G-e,d(e))$ determines
$G$.
\end{lemma}

For convenience, we use $D_{m,n,p}^{s,t}$ to denote the graph
obtained from $D_{m,n,p}$ by subdividing a left leaf edge $s$
times and a right leaf edge $t$ times. Figure~\ref{f1} shows the
graph $D_{3,4,5}^{1,2}$.

\begin{figure}[h]
\begin{center}
\begin{pspicture}(1.5,2.0)(7.5,5)
\psset{radius=.09, xunit=1.5, yunit=1.4}

\Cnode(2,2.5){0} \Cnode(2.5,2.5){1}
\Cnode(3.0,2.5){2}\Cnode(3.5,2.5){3}\Cnode(4.0,2.5){4}

\Cnode(1,2.5){5} \Cnode(1,1.8){6}
\Cnode(1,3.27){8}\Cnode(1.5,2.9){7}

\Cnode(5,2.25){11} \Cnode(5,1.8){12}
\Cnode(5,3.26){13}\Cnode(5,2.7){14}

\Cnode(4.36,2.74){15}\Cnode(4.7,3.0){16}

 \ncline{4}{11}\ncline{4}{12}\ncline{4}{14}
 \ncline{15}{16}\ncline{13}{16}\ncline{15}{4}
\ncline{0}{1}\ncline{1}{2}\ncline{2}{3}\ncline{3}{4}\ncline{0}{5}
\ncline{0}{6}\ncline{0}{7}\ncline{7}{8}
\end{pspicture}
\caption{\label{f1}\footnotesize {The graph $D_{3,4,5}^{1,2}$.}}
\end{center}
\end{figure}
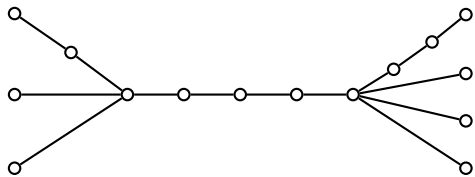

\begin{lemma} \label{hub}
If $p\ge4$, $m\geq 2$, $n\geq 3$, and $n\ne m+1$, then two hub
decards determine $D_{m,n,p}$ except $D_{2,n,5}$. If
$p\in\{2,3\}$, or $m=1, n\ge4$, or $m=1, n=3, p=4$, one hub
decard determines $D_{m,n,p}$.
\end{lemma}

\begin{proof}
Lemma~\ref{leme1} implies the second conclusion is true.  Assume
$p\geq 3$ and $m\ge2$ in the following. We are considering the
two hub decards $(K_{1,m}+B_{n,p-1},m+1)$ and
$(B_{m,p-1}+K_{1,n},n+1)$.

Let $G$ be a graph that arises from $B_{m,p-1}+K_{1,n}$ by
adding an edge $e$ of degree $n+1$.  Note that there are one
$m+1$-vertex and one $n$-vertex in $B_{m,p-1}+K_{1,n}$.

If $n\notin \{3, m+1,m+2\}$, then $G$ is obtained by adding an
edge $e$ joining the $n$-vertex in $K_{1,n}$ and a $1$-vertex in
$B_{m,p-1}$. If $G\ncong D_{m,n,p}$, then $p\geq 4$ and $G\cong
D_{m,n,3}^{p-3,0}$. But deleting an $(m+1)$-edge other than $e$
from $G$ cannot arise $K_{1,m}$ unless $m=2$ and $p=5$. Note
that  $D_{2,n,3}^{2,0}$ and $D_{2,n,5}$ share the decards
$(K_{1,2}+B_{n,4},3)$ and $(B_{2,4}+K_{1,n},n+1)$.

If $n=m+2$ and $G$ is obtained from $B_{m,p-1}+K_{1,m+2}$ by
adding an edge $e$ joining the $(m+1)$-vertex and a $2$-vertex
in $B_{m,p-1}$, then $G$ has a cycle $C$  and $p\geq 5$. The
graph obtained by deleting any edge $e'$ on $C$ other than $e$
from $G$ has $K_{1,m+2}$ as a component, which cannot happen in
the edge-card $K_{1,m}+B_{n,p-1}$.

If $m=n=3$ and $G$ arises from $B_{3,p-1}+K_{1,3}$ by adding an
edge $e$ joining two non-adjacent $2$-vertices, then $p\geq6$,
and $G$ has a cycle $C$ and the degree of any edge other than
$e$ on the cycle $C$ is less than $4$. Hence, $G$ shares only
one hub decard with $D_{3,3,p}$.
\end{proof}

\begin{lemma} \label{middle}
Any two distinct middle decards determine $D_{m,n,p}$ except
$1=m<n$ and $p\ge6$. In the special case three middle decards
determine it.
\end{lemma}

\begin{proof}
Assume two different middle decards are $(B_{m,a}+B_{n,p-a},2)$
and $(B_{m,b}+B_{n,p-b},2)$ with $a< b$. Any component of
$B_{m,a}+B_{n,p-a}$ is different with any component of
$B_{m,b}+B_{n,p-b}$.  Let $G$ be the graph arises from edge-card
$B_{m,b}+B_{n,p-b}$ by adding an edge $e$ joining two
non-adjacent $1$-vertices.

If $e$  joins two $1$-vertices in one component, then $G$ has a
cycle $C$. The graph obtained by deleting any edge $e'$ on $C$
other than $e$ from $G$ has either $B_{m,b}$ or $B_{n,p-b}$ as a
component which is not a component of $B_{m,a}+B_{n,p-a}$.
Hence, $G$ shares only one middle decard with $D_{m,n,p}$.

Assume $e$ joins two $1$-vertices in two components
respectively. We discuss the cases according to $m$.

{\bf Case 1:}  {\it $m=1$.}

If $n=1$ or $n>1$ and $p-b=2$, then the reconstruction $G$ is
$D_{1,n,p}$. Hence we assume $n>1$ and $p-b>2$ in the following.
If $p=5$ and $G\ncong D_{1,n,p}$, then $G\cong D_{1,n,4}^{0,1}$.
The graph obtained by deleting any $2$-edge $e'$ other than $e$
from $G$ has only one broom component. Hence $G$ and $D_{1,n,5}$
shares one middle decard. If $p\geq 6$ and $G\ncong D_{1,n,p}$,
then $G\cong D_{1,n,b+2}^{0,p-b-2}$. If $p-b=3$, the graph
obtained by deleting any $2$-edge $e'$ other than $e$ from $G$
has only one broom component. Hence, $G$ and $D_{1,n,p}$ shares
one middle decard. Assume $p-b\geq 4$. Let $e'$ be another
$2$-edge adjacent to another $(n+1)$-edge. Moreover, $e'$ is not
a leaf edge when $n=2$. Then $G-e'\cong B_{1,p-b-3}+B_{n,b+3}$.
The graph obtained by deleting any $2$-edge $e''$ other than $e$
and $e'$ from $G$ has only one broom component. Hence, $G$ and
$D_{1,n,p}$ shares at most two middle decards.

{\bf Case 2:}  {\it $m\geq 2$.}

If $G\ncong D_{m,n,p}$, $G$ has three possibilities
$D_{m,n,4}^{b-2,p-b-2}$,  $D_{m,n,b+2}^{0,p-b-2}$ or
$D_{m,n,p-b-2}^{b-2,0}$. The graph obtained by deleting any
$2$-edge $e'$ other than $e$ from $G$ has only one broom
component. Hence, $G$ shares one middle decard with $D_{m,n,p}$.
\end{proof}

\begin{lemma} \label{oneleaf}
Any two distinct leaf decards determine $D_{m,n,p}$ except
$D_{1,2,p}$($p\geq 3$), and any three leaf decards suffice except
$D_{m,m+2,3}$.
\end{lemma}

\begin{proof}
Let $\hat{e}$ be a leaf edge of $D_{m,n,p}$ and
$H=D_{m,n,p}-\hat{e}$. Then $H$ contains a double broom
component $H'$ and an isolated vertex. The diameter of $H'$ is
$p+1$ when $d(\hat{e})\geq 2$, and $p$ when $d(\hat{e})=1$. We
reconstruct $G$ from the leaf card $H$.

If $m\geq 5$, or $m=4$ and $p\leq 4$, or $m\in \{1,3\}$ and
$p=2$, or $m=n=1$,  Lemma~\ref{leme1} implies that a left leaf
decard determines $D_{m,n,p}$.

If $n\geq 5$ and $m\notin \{n-2,n-3\}$, or $m+3=n\geq 5$ and
$p\leq 3$, or $m+1=n=4$ and $p\leq 4$, or $n=4, m=1$ and $p\leq
3$, or $m+1=n=3$ and $p=2$, Lemma~\ref{leme1} implies that a
right leaf decard determines $D_{m,n,p}$.

If $n=m+2$ and $p=3$, the right leaf decard is
$(D_{m,m+1,3}+K_1,m+2)$. Let $G_1$ be the graph obtained from a
cycle $C_4$ by appending $m$-vertices to two non-adjacent
vertices respectively(see Fig.\ref{f2}). The degree of any edge
on the cycle is $m+2$. Let $G'=G_1+K_1$ and the edger-card
obtained by deleting any edge on the cycle from $G'$ is
isomorphism to $D_{m,m+1,3}+K_1$. Then $G'$ shares
$\min\{4,m+2\}$ right leaf decards with $D_{m,m+2,3}$. Hence,
three leaf decards cannot determine $D_{m,m+2,3}$.

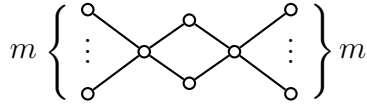
\begin{figure}[h]
\begin{center}
\begin{pspicture}(5.5,1.6)(6.5,3)
\psset{radius=.09, xunit=1.5, yunit=1.4}

\Cnode(4,1.2){0} \Cnode(4,1.8){1}
\Cnode(3.6,1.5){2}\Cnode(4.4,1.5){3}

\Cnode(3.1,1.9){4}\Cnode(3.1,1.1){5}
\Cnode(4.9,1.9){6}\Cnode(4.9,1.1){7}

\rput(3.1,1.58){$\vdots$}\rput(4.9,1.58){$\vdots$}
\rput(5.45,1.5){$m$}\rput(2.53,1.5){$m$}

\rput(5.18,1.5){$\Bigg\}$}\rput(2.83,1.49){$\Bigg\{$}

\ncline{0}{2} \ncline{0}{3} \ncline{1}{2} \ncline{1}{3}
\ncline{4}{2}\ncline{5}{2}\ncline{6}{3}\ncline{7}{3}
\end{pspicture}
\caption{\label{f2}\footnotesize {The graph $G_1$.}}
\end{center}
\end{figure}

Otherwise, we consider the following cases according to $d(\hat e)$.

If $d(\hat{e})=1$, then $m=1$, $n\geq 2$ and $p\geq 3$.  The
graph $G$ is obtained from $H$ by adding an edge $e$ joining a
$1$-vertex and an isolated vertex. If $G\ncong D_{1,n,p}$, then
$G\cong D_{1,n,p-1}^{0,1}$. If $n\geq3$, the graph obtained from
$G$ by deleting any edge $e'$ of degree $n$ from $G$ has no
double-broom component. Then $G$ shares only one leaf decard
with $D_{1,n,p}$. If $n=2$, let $e'$ be the $2$-edge incident to
a leaf, then $G-e'\cong D_{1,1,p}+K_{1}$. The graph obtained
from $G$ by deleting any edge $e''$ other than $e$ and $e'$ from
$G$ cannot contain both a double-broom component and an isolated
vertex. Hence, $G$ shares exactly two distinct leaf decards with
$D_{1,2,p}$.

If $d(\hat{e})=2$, then $m=2$ or $n=2$. The graph $G$ arises
from $H$ by joining two non-adjacent vertices whose degree sum
is $2$. Then $G$ is obtained from $H$ by adding an edge $e$
joining two $1$-vertices or a $2$-vertex and an isolated vertex.
If $e$ joins two $1$-vertices, $G$ has a cycle $C_l$. If $m=1$
and $n=2$, $G\cong C_{p+2}+K_1$. In this case $G$ shares two
right leaf decards with $D_{1,2,p}$. If $m=2$ and $l\geq 4$, let
$e'$ be another edge adjacent to another $(n+1)$-edge on the
cycle $C$, then $G-e'\cong H$. The graph obtained by deleting
any edge $e''$ other than $e$ and $e'$ on the cycle has no
double-broom component. Hence $G$ shares exactly two left leaf
decards with $D_{2,n,p}$. If $m=2$ and $l=3$, the degree of any
edge on the cycle other than $e$ is $n+1$. Hence $G$ shares one
leaf decard with $D_{2,n,p}$. Assume $e$ joins a $2$-vertex and
an isolated vertex, and  $G\ncong D_{m,n,p}$. If $G\cong
D_{1,2,p-1}^{0,1}$, then $G$ shares exactly two distinct leaf
decards with $D_{1,2,p}$ as discussed above. If $G\ncong
D_{1,2,p-1}^{0,1}$, the graph obtained by deleting any edge $e'$
of degree $m$ or $n$ cannot contain both a double-broom
component and  an isolated vertex. Then $G$ shares only one leaf
decards with $D_{m,n,p}$.

If $d(\hat{e})=3$, the graph $G$ arises from $H$ by joining a
$1$-vertex $u$ and a $2$-vertex $v$. If $G\ncong D_{m,n,p}$,
then $G$ has a cycle $C_l$. Assume the degree of any vertex on
$C_l$ other than $v$ is $2$ in $G$. Let $e'$ be another edge
other than $e$ incident to $v$ on the cycle $C_l$. Then
$G-e'\cong H$. The graph obtained by deleting any edge $e''$
other than $e$ and $e'$ on the $C_l$ from $G$ has no
double-broom component or a double-broom component with diameter
$p$. Then $G$ shares exactly two identical leaf decards with
$D_{m,n,p}$. Assume there are two vertices with degree at least
$3$ on $C_l$ in $G$.  Suppose $l=4$, let $e'$ be the other edge
incident to $v$ on the cycle $C_l$. Then $G-e'\cong H$. The
graph obtained by deleting any edge $e''$ other than $e$ and
$e'$ on the $C_l$ from $G$ has no double-broom component. Then
$G$ shares exactly two identical leaf decards with $D_{m,n,p}$.
If $l=p+1$, then $m=1$ and $n=3$. Assume $p\neq 3$. Let $e'$ be
the other edge incident to $u$ on the cycle $C_l$. Then
$G-e'\cong H$. The graph obtained by deleting any edge $e''$
other than $e$ and $e'$ on the $C_l$ from $G$ has no
double-broom component or a double-broom component with diameter
$p+2$ or $p$. Then $G$ shares exactly two right leaf decards
with $D_{1,3,p}$. If $l\notin \{4,p+1\}$, the graph obtained by
deleting any edge $e'$ other than $e$ on the $C_l$ from $G$ has
no double-broom component or a double-broom component with
diameter $p+2$. Then $G$ shares only one leaf decard with
$D_{m,n,p}$.

If $d(\hat{e})=4$, the graph $G$ arises from $H$ by joining two
non-adjacent vertices whose degree sum is $4$. If $G\ncong
D_{m,n,p}$, then $G$ is obtained from $H$ by adding an edge $e$
joining two non-adjacent $2$-vertices or a $1$-vertex $u$ and a
$3$-vertex. Hence $G$ has a cycle $C$. If $e$ joins two two
non-adjacent $2$-vertices, the graph obtained by deleting any
edge $e'$ other than $e$ on the $C$ from $G$ has no double-broom
component or has a double-broom component with diameter $p$ when
$m=1$. Then $G$ shares one leaf decard with $D_{m,n,p}$. If $e$
joins a $1$-vertex $u$ and a $3$-vertex, then $m=2,n=4$. Assume
$p\neq 3$. Let $e'$ be the other edge incident to $u$ on the
cycle $C$, then $G-e'\cong H$. The graph obtained by deleting
any edge $e''$ other than $e$ and $e'$ on the $C$ from $G$ has
no double-broom component or  a double-broom component with
diameter $p+2$ or $p$. Then $G$ shares exactly two leaf decards
with $D_{2,4,p}$.

If $5\leq d(\hat{e})=m+2$ and $p\neq 3$, the graph $G$ arises
from $H$ by joining two non-adjacent vertices whose degree sum
is $m+2$. If $G\ncong D_{m,m+2,p}$, then $G$ is obtained from
$H$ by adding an edge $e$ joining a $1$-vertex $u$ and the
$(m+1)$-vertex. Then $G$ has a cycle $C$. Let $e'$ be another
edge incident to $u$ on the cycle $C$, then $G-e'\cong H$. The
graph obtained by deleting any edge $e''$ other than $e$ and
$e'$ on the $C$ from $G$ has no double-broom component or a
double-broom component with diameter $p+2$ or $p$. Then $G$
shares exactly two right leaf decards with $D_{m,m+2,p}$.

If $5\leq d(\hat{e})=m+3$ and $p\geq 4$, the graph $G$ arises
from $H$ by joining two non-adjacent vertices whose degree sum
is $m+3$. If $G\ncong D_{m,m+3,p}$, then  $G$ is obtained from
$H$ by adding an edge $e$ joining  a $2$-vertex and the
$(m+1)$-vertex. Then $G$ has a cycle $C$. The graph obtained by
deleting any edge $e'$ other than $e$ on the cycle $C$ from $G$
has no double-broom component or a double-broom component with
diameter $p$. Then $G$ shares only one right leaf decards with
$D_{m,n,p}$.
\end{proof}

\begin{lemma}\label{leafhub}
A leaf decard and a hub decard determine $D_{m,n,p}$ except
$p\geq 4$, $m=2$ or $n=2$.
\end{lemma}

\begin{proof}
Let $\hat{e}$ be a leaf edge of $D_{m,n,p}$ and
$H=D_{m,n,p}-\hat{e}$. Then $H$ contains an isolated vertex. We
reconstruct $G$ from the leaf card $H$.

If $d(\hat{e})\geq 3$ and $G\ncong D_{m,n,p}$, then $G$ has a
cycle and an isolated vertex. Since any hub edge-card has no
isolated vertex, $G$ shares no hub decard with $D_{m,n,p}$. If
$d(\hat{e})=2$, then $m=2$ or $n=2$. By the assumption, $p\leq
3$. By Lemma~\ref{hub}, one hub decard is sufficient. If
$d(\hat{e})=1$, by Lemma~\ref{hub} and the proof of
Lemma~\ref{oneleaf}, we only need consider $m=1,n=3,p\geq5$. If
$G\ncong D_{1,3,p}$, then $G\cong D_{1,3,p-1}^{0,1}$.  The graph
obtained by deleting any edge from $G$ cannot have $K_{1,3}$
component. Since the hub decard is $(B_{1,p-1}+K_{1,3},4)$,  $G$
shares no hub decard with $D_{1,3,p}$.
\end{proof}

\begin{lemma}\label{leafmiddle}
A leaf decard and a middle decard determine $D_{m,n,p}$, except
$m=2, p\geq 5$ and $m=1, n\geq 2, p\geq 4$.
\end{lemma}

\begin{proof}
Let $\hat{e}$ be a leaf edge of $D_{m,n,p}$ and
$H=D_{m,n,p}-\hat{e}$. Then $H$ contains an isolated vertex. We
reconstruct $G$ from the leaf card $H$. If $G$ has a cycle, then
$G$ has an isolated vertex. Since any middle edge-card has no
isolated vertex, $G$ shares no middle decard with $D_{m,n,p}$.
Assume $G$ has no cycle in the following. That is, $G$ is
obtained from $H$ by adding an edge $e$ joining the isolated
vertex and a $d(\hat{e})$-vertex.

If $d(\hat{e})\geq 3$, then $G\cong D_{m,n,p}$.

If $d(\hat{e})=2$, then $G$ is obtained from $H$ by adding an
edge $e$ joining the isolated vertex and a $2$-vertex. If $m=2$,
$p=4$ and $G\ncong D_{2,n,4}$, $G$ is $D_{2,n,3}^{1,0}$ or
$D_{2,n,2}^{2,0}$. The graph obtained by deleting any  $2$-edge
$e'$ from $G$ has an isolated vertex or a path. Hence, $G$
shares no middle decard with $D_{2,n,4}$. If $m=1,n=2,p=3$ and
$G\ncong D_{1,2,3}$, then $G\cong D_{1,2,2}^{0,1}$. There is no
$2$-edge other than $e$ in $G$. Hence, $G$ shares no middle
decard with $D_{1,2,3}$.

If $d(\hat{e})=1$, then $m=1,n\geq2,p=3$ by the assumption. The
graph $G$ is obtained from $H$ by adding an edge $e$ joining a
leaf and the isolated vertex. If $G\ncong D_{1,n,3}$, then
$G\cong D_{1,n,2}^{0,1}$. $G$ has no $2$-edge or a
 $2$-edge incident to a leaf when $n=2$. Hence, $G$ shares no middle
decard with $D_{1,n,3}$.
\end{proof}

\begin{lemma}\label{hubmiddle}
A hub decard and a middle decard determine $D_{m,n,p}$, except
$m=n-1$ and $p\geq 4$, or $m=2$ and $p\geq 6$.
\end{lemma}

\begin{proof} By Lemma~\ref{hub},
one hub decard determines $D_{m,n,p}$ when $p\in \{2,3\}$.
Assume $p\geq4$. Let $H=B_{m,a}+B_{n,p-a}$ be the given middle
card. The graph $G$ is obtained from $H$ by adding an edge $e$
joining two non-adjacent $1$-vertices.

If $e$  joins two $1$-vertices in one component, we have a cycle
$C$. If $m\ne n-1$, the graph obtained by deleting any edge on
$C$ has no $K_{1,m}$ or $K_{1,n}$ as a component. Since any hub
edge-card has $K_{1,m}$ or $K_{1,n}$ as a component, $G$ shares
no hub decard with $D_{m,n,p}$.

Assume $e$ joins two $1$-vertices in two components
respectively. We discuss the cases according to $m$. If $m=1$,
by Lemma~\ref{hub}, we only need consider $m=1,n=3,p\geq5$. If
$G\ncong D_{1,3,p}$, then $G$ is obtained from a $(p+2)$-vertex
path by adding  $2$ neighbors to a vertex at distance at least
$2$ from its end. The graph obtained by deleting any edge from
$G$ cannot have $K_{1,3}$ component. Since the hub decard is
$(B_{1,p-1}+K_{1,3},4)$,  $G$ shares no hub decard with
$D_{1,3,p}$. If $m=2$, by the assumption, $p\in \{4,5\}$. If
$m=2, p=4$,  then the reconstruction $G$ is $D_{2,n,4}$. If
$m=2, p=5$ and $G \ncong D_{2,n,5}$, then $G$ is
$D_{2,n,4}^{0,1}$ or $D_{2,n,4}^{1,0}$. The graph obtained by
deleting a $(m+1)$-edge or $(n+1)$-edge from $G$ cannot contain
both a  broom component and a $K_{1,2}$ or $K_{1,n}$ component.
$G$ shares no hub decard with $D_{2,n,5}$. If $m\geq 3$ and
$G\ncong D_{m,n,p}$, $G$ has three possibilities
$D_{m,n,4}^{a-2,p-a-2}$,  $D_{m,n,a+2}^{0,p-a-2}$ or
$D_{m,n,p-a-2}^{a-2,0}$. The graph obtained by deleting a
$(m+1)$-edge or $(n+1)$-edge from $G$ cannot contain both a
broom component and a $K_{1,m}$ or $K_{1,n}$ component. Hence,
$G$ shares no hub decard with $D_{m,n,p}$.
\end{proof}

\section{Main Results} \label{sec:results}

\begin{theorem}
\label{main}
For the double-broom $D_{m,n,p}$ with $1 \leq m \leq n$ and $p\geq 2$,
$$\dern(D_{m,n,p})=
 \begin{cases}
    1,   & \textrm{if there is an edge satisfying the condition of Lemma~\ref{leme1};}\\
    2,   & \textrm{otherwise.}
\end{cases}$$
\end{theorem}

\begin{proof}
The decard satisfying the condition of Lemma~\ref{leme1}
determines it. Otherwise, any decard cannot determine it, hence
$\dern(D_{m,n,p})\ge 2$. By the proofs of
Lemmas~\ref{middle},~\ref{oneleaf},~\ref{leafhub},
\ref{leafmiddle} and \ref{hubmiddle}, there are two decards
determine $D_{m,n,p}$. Hence, we have $\dern(D_{m,n,p})\leq 2$.
\end{proof}

\begin{theorem}
\label{main}
For the double-broom $D_{m,n,p}$ with $1 \leq m \leq n$ and $p\geq 2$,
$$\adern(D_{m,n,p})=
 \begin{cases}
    1,   & \textrm{for $D_{m,n,2} (2\notin \{m,n\}$ and $m\neq n-2$),~or~$D_{m,n,3}$($m\geq 4$ and $m\neq n-2$);}\\
    5,   & \textrm{for $D_{1,2,4}$, $D_{1,2,p}(p\geq 6)$,~or~$D_{m,m+2,3}(m\geq 2$);}\\
    4,   & \textrm{for $D_{1,2,5}$, $D_{1,3,3}$~or~$D_{2,n,p}(p\geq5)$;}\\
    3,   & \textrm{for $D_{1,1,p}(p\geq 3)$, $D_{1,2,p}(p\leq 3)$, $D_{1,n,p}(n\geq 4, p\geq 4)$,}\\
    & \textrm{ $D_{2,n,p}(p\leq4, \mbox{except}\ D_{2,4,3} )$, $D_{3,n,p}(p\geq 4)$, $D_{m,m,p}(m\geq4 ,p\geq 5)$,}\\
    & \textrm{ $D_{m,m+1,p}(m\geq 4,p\geq 4)$, or~$D_{m,m+2,p}(m\notin \{2,3\}, p\neq 3)$,or~$D_{3,5,2}$;}\\
    2,   & \textrm{otherwise.}
\end{cases}$$
\end{theorem}

\begin{proof} In arguments for upper bounds, we will always let $\cS$ be a given
list of decards of $D_{m,n,p}$ and $G$ be a reconstruction from
$\cS$. If this forces $G\cong D_{m,n,p}$ wherever $|S|=k$, then
$\adern(D_{m,n,p})\leq k$, but a single exception yields
$\adern(D_{m,n,p})\ge k$.

{\it Case 1: Every edge of $D_{m,n,p}$ satisfies the condition
of Lemma~\ref{leme1}.} By Lemma~\ref{leme1}, every decard
determines it.

{\it Case 2: The graph has three possibilities.}

{\it The graph is $D_{1,2,4}$.} The graphs $D_{1,2,4}$ and
$D_{1,2,2}^{0,2}$ have four common decards: $(D_{1,2,3},1)$,
$(D_{1,1,4},2)$, $(B_{1,1}+B_{2,3},2)$, and
$(B_{1,3}+K_{1,2},3)$. Hence, $\adern(D_{1,2,4})\geq 5$.
For the upper bound, if $|\cS|=5$,
then $\cS$ has either $2$ middle decards
or $3$ leaf decards. By Lemmas~\ref{middle} and~\ref{oneleaf}, the
reconstruction $G$ is $D_{1,2,4}$. We have
$\adern(D_{1,2,4})\leq 5$.

{\it The graph is $D_{1,2,p}$ and $p\geq 6$.} The graphs
$D_{1,2,p}$ and $D_{1,2,p-2}^{0,2}$ have four common decards:
$(B_{1,1,p}+K_{1},2)$, $(B_{1,p-1}+K_{1,2},3)$,
$(B_{1,1}+B_{2,p-1},2)$, and $(B_{1,p-4}+B_{2,4},2)$. For the
upper bound, assume $|\cS|=5$. By Lemmas~\ref{middle}
and~\ref{oneleaf}, $3$ leaf decards or $3$ middle decards
determine it. Assume $\cS$ has $2$ leaf decards and $2$ middle
decards. Since $D_{1,2,p}$ has one left leaf decard, $\cS$
contains one right leaf decard $(B_{1,1,p}+K_1,2)$. The
reconstruction from $(B_{1,1,p}+K_1,2)$ other than $D_{1,2,p}$
is $C_{p+2}+K_1$ or $B_{1,2,a}^{0,p-a}(2\leq a\leq p-2)$. But
each has at most four decards in $\cS$. We have
$\adern(D_{1,2,p})\leq 5 (p\geq 6)$.

{\it The graph is $D_{m,m+2,3}(m\geq 2).$} By the proof of
Lemma~\ref{oneleaf}, $D_{m,m+2,3}$ and $G_1+K_1$ have four
common right leaf decards when $m\geq 2$. For the upper bound,
assume $|\cS|=5$. By Lemmas~\ref{hub} and~\ref{oneleaf}, one hub
decard or $3$ leaf decards containing at least one left leaf
decard determine it. Assume $\cS$ has $5$ right leaf decards
$(D_{m,m+1,3}+K_1, m+2)$. The reconstruction from
$(D_{m,m+1,3}+K_1, m+2)$ other than $D_{m,m+2,3}$ is $G_1+K_1$.
But there are four  decards $(D_{m,m+1,3}+K_1, m+2)$ in the
dedeck of $G_1+K_1$. We have $\adern(D_{m,m+2,3})\leq 5 (m\geq
2)$.

{\it Case 3: The graph has three possibilities.}

{\it The graph is $D_{1,2,5}$.} The graphs $D_{1,2,5}$ and
$D_{1,2,3}^{0,2}$ have three common decards:
$(B_{1,1,5}+K_{1},2)$, $(B_{1,1}+B_{2,4},2)$, and
$(B_{1,4}+K_{1,2},3)$.  For the upper bound, assume $|\cS|=4$.
By Lemmas~\ref{middle} and~\ref{oneleaf}, we may assume $\cS$
has $2$ leaf decards, $1$ middle decard and $1$ hub decard.
Since $D_{1,2,5}$ has one left leaf decard, $\cS$ contains one
right leaf decard $(B_{1,1,5}+K_1,2)$. The reconstruction from
$(B_{1,1,p}+K_1,2)$ other than $D_{1,2,5}$ is $C_{7}+K_1$ or
$B_{1,2,a}^{0,5-a}(2\leq a\leq 3)$. But each has at most three
decards in $\cS$. We have $\adern(D_{1,2,5})\leq 4$.

{\it The graph is $D_{1,3,3}.$} By the proof of
Lemma~\ref{oneleaf}, $D_{1,3,3}$ and $G_1+K_1$ have three common
right leaf decards: $(D_{1,2,3}+K_1, 3)$. For the upper bound,
assume $|\cS|=4$. Since there are $6$ decards in its dedeck,
$\cS$ has the hub decard or one middle decard and a leaf decard
or two distinct leaf decards. By
Lemmas~\ref{hub},~\ref{oneleaf}, and~\ref{leafmiddle},  the
reconstruction is $D_{1,3,3}$. We have $\adern(D_{1,3,3})\leq
4$.

{\it The graph is $D_{2,n,p}(p\geq 5).$} The graphs $D_{2,n,5}$
and $D_{2,n,3}^{2,0}$ have three common decards:
$(B_{1,n,5}+K_{1},2)$, $(K_{1,2}+B_{n,4},3)$, and
$(K_{1,n}+B_{2,4},n+1)$. The graphs $D_{2,n,p}(p\geq 6)$ and
$D_{2,n,p-2}^{2,0}$ have three common decards:
$(B_{1,n,p}+K_{1},2)$, $(B_{2,4}+B_{n,p-4},2)$, and
$(B_{1,p-1}+K_{1,2},3)$. If $|\cS|=4$, by Lemmas~\ref{middle}
and~\ref{oneleaf}, we need consider the following cases. If
$\cS$ has $2$ identical leaf decards and other kinds of decards,
the reconstruction $G$ from a leaf decard having two identical
leaf decards other than $D_{2,n,p}$  has an isolated vertex and
shares no other kinds of decard with $D_{2,n,p}$. If $\cS$ has
$2$ identical middle decards and other kinds of decards, then
$n=2$. Assume the middle decard is $B_{2,a}+B_{2,p-a}$. The
reconstruction from $B_{2,a}+B_{2,p-a}$ having two identical
middle decards other than $D_{2,2,p}$ has $B_{2,a}$ or
$B_{2,p-a}$ as a component and shares no other kinds of decard
with $D_{2,2,p}$. If $\cS$ has $2$ hub decards and other kinds
of decards, by Lemmas~\ref{hub}, we may assume $n=2$, or $n=3$,
or $n\geq 4$ and $p=5$. When $n=2$, the reconstruction $G$ from
a hub decard having two hub decards other than $D_{2,2,p}$ has
$K_{1,2}$ as a component or is $D_{2,2,3}^{2,0}$ when $p=5$. But
$D_{2,2,3}^{2,0}$ shares three decards with $D_{2,2,5}$,
otherwise $G$ shares no other kinds of decard with $D_{2,2,p}$.
When $n=3$, the reconstruction from $(B_{2,p-1}+K_{1,3},4)$
other than $D_{2,3,p}$ is $D_{2,3,3}^{0,p-3}$, or
$D_{2,3,3}^{p-3,0}$. But each of them shares at most three
decards with $D_{2,3,p}$($D_{2,3,3}^{2,0}$ shares three decards
with $D_{2,3,5}$). When $n\geq 4$ and $p=5$, the reconstruction
from $(B_{2,4}+K_{1,n},n+1)$ other than $D_{2,n,5}$ is
$D_{2,n,3}^{2,0}$. But $D_{2,n,3}^{2,0}$ shares three decards
with $D_{2,n,5}$. We have $\adern(D_{2,n,p})\leq 4 (p\geq 5)$.

{\it Case 4: The graph has nine possibilities.}

{\it The graph is $D_{1,1,p}(p\geq 3)$.} The graphs $D_{1,1,p}$
and $C_{p}+B_{1,1}$ have two common decards:
$(B_{1,1}+B_{1,p-1},2)$. If $|\cS|=3$, then $\cS$ has $2$
distinct middle decards or $1$ leaf decard. By
Lemmas~\ref{leme1} and~\ref{middle}, the reconstruction $G$ is
$D_{1,1,p}$. We have $\adern(D_{1,1,p})\leq 3 (p\geq 3)$.

{\it The graph is $D_{1,2,p}(p\leq 3)$.} The graphs $D_{1,2,p}$
and $C_{p+2}+K_{1}$ have two common decards:
$(B_{1,1,p}+K_{1},2)$. If $|\cS|=3$, $\cS$ has $1$ hub decard,
or $3$ leaf decards, or $2$ leaf decards and $1$ middle decard.
By Lemmas~\ref{hub},~\ref{oneleaf}, and~\ref{leafmiddle}, $G$ is
$D_{1,2,p}$. We have $\adern(D_{1,2,p})\leq 3 (p\leq 3)$.

{\it The graph is  $D_{1,n,p}(n\geq 4, p\geq 4)$.} The graphs
$D_{1,n,p}$ and $D_{1,n,p-1}^{0,1}$ have two common decards:
$(B_{1,n,p-1}+K_{1},1)$ and $(B_{1,p-3}+B_{n,3},2)$. For the
upper bound, by the proofs of
Lemmas~\ref{hub},~\ref{middle},~\ref{oneleaf},~\ref{leafhub},~\ref{leafmiddle},
and~\ref{hubmiddle}, we may assume $\cS$ has $2$ middile decard
and $1$ leaf decard. Assume $\cS$ has $(B_{1,a}+B_{n,p-a},2)$.
The reconstruction from $(B_{1,a}+B_{n,p-a},2)$ other than
$D_{1,n,p}$ is $D_{1,n,a+2}^{0,p-a-2}$. But
$D_{1,n,a+2}^{0,p-a-2}$ shares at most two decards with
$D_{1,n,p}$. We have $\adern(D_{1,n,p})\leq 3 (n\geq 4, p\geq
4)$.

{\it The graph is  $D_{2,n,p}(p\leq4\  \mbox{except}\
D_{2,4,3})$.} Let $G_2$ be the graph obtained from
$C_{p+2}+K_{1}$ by appending $(n-1)$ vertices to one vertex on
the cycle. The graphs $D_{2,n,p}$ and $G_2$ have two common
decards: $(B_{1,n,p}+K_{1},2)$. If $|\cS|=3$, by the proofs of
Lemmas~\ref{hub},~\ref{oneleaf}, and~\ref{leafmiddle}, we may
assume $\cS$ has a hub decard and $p=4$. The reconstruction from
$(K_{1,2}+B_{n,3},3)$ or $(K_{1,n}+B_{2,3},n+1)$ other than
$D_{2,n,4}$ is $D_{2,n,3}^{0,1}$, $D_{2,n,3}^{1,0}$, or has a
$K_{1,2}$ component. But anyone of them shares at most two
decards with $D_{2,n,p}$. We have $\adern(D_{2,n,p})\leq 3
(p\leq 4\  \mbox{except}\ D_{2,4,3})$.

{\it The graph is  $D_{3,n,p}(p\geq 4)$.} Let $G_3$ be the graph
obtained from  $D_{2,n,p}+K_1$ by joining a $1$-vertex $u$ to
the vertex at distance $3$ from $u$. The graphs $D_{3,n,p}$ and
$G_3$ have two common left leaf decards: $(B_{2,n,p}+K_{1},3)$.
If $|\cS|=3$ and $n\neq 4$, by
Lemmas~\ref{leafhub},~\ref{leafmiddle}, and~\ref{hubmiddle}, we
may assume the decards in $\cS$ are all leaf decards or middle
decards. By Lemmas~\ref{middle} and~\ref{oneleaf}, the
reconstruction $G$ is $D_{3,n,p}$. If $n=4$, by
Lemmas~\ref{middle},~\ref{oneleaf},~\ref{leafhub},
and~\ref{leafmiddle}, assume $\cS$ has $2$ hub decards and $1$
middle decard. The reconstruction from a hub decard other than
$D_{m,m+1,p}$ is $D_{3,4,3}^{0,p-3}$ or $D_{3,4,3}^{p-3,0}$. But
each of them shares at most two decards with $D_{3,4,p}$. We
have $\adern(D_{3,n,p})\leq 4 (p\geq 4)$.

{\it The graph is  $D_{m,m,p}(m\geq 4,p\geq 5)$.} Let $G_4$ be
the graph obtained from $C_{p-1}+B_{m,2}$ by appending $(m-1)$
vertices to one vertex on the cycle. The graphs $G_4$ and
$D_{m,m,p}$ have two common decards: $(B_{m,2}+B_{m,p-2},2)$. If
$|\cS|=3$, by Lemmas~\ref{leafhub},~\ref{leafmiddle},
and~\ref{hubmiddle}, we may assume the decards in $\cS$ are all
leaf decards or middle decards. By Lemmas~\ref{middle}
and~\ref{oneleaf}, the reconstruction $G$ is $D_{m,m,p}$. We
have $\adern(D_{m,m,p})\leq 3 (m\geq 4,p\geq 5)$.

{\it The graph is $D_{m,m+1,p}(m\geq 4,p\geq 4)$.} The graphs
$D_{m,m+1,3}^{0,p-3}$ and $D_{m,m+1,p}$ have two common decards:
$(K_{1,m}+B_{m+1,p-1},m+1)$ and $(K_{1,m+1}+B_{m,p-1},m+2)$.
Hence, $\adern(D_{m,m+1,p})\geq 3$. By
Lemmas~\ref{middle},~\ref{oneleaf},~\ref{leafhub},
and~\ref{leafmiddle}, assume $\cS$ has $2$ hub decards and $1$
middle decard.  The reconstruction from a hub decard other than
$D_{m,m+1,p}$ is $D_{m,m+1,3}^{0,p-3}$ or $D_{m,m+1,3}^{p-3,0}$.
But each of them shares at most two decards with $D_{m,m+1,p}$.
We have $\adern(D_{m,m+1,p})\leq 3 (m\geq 4,p\geq 4)$.

{\it The graph is $D_{m,m+2,p}(m\notin \{2,3\} , p\neq 3)$ or
$D_{3,5,2}$.} Let $G_5$ be a graph obtained from $C_{p+1}+K_1$
by appending $m$ vertices to two vertices having a common
neighbor in $C_{p+1}$ respectively. The graphs $D_{m,m+2,p}$ and
$G_5$ have two common right leaf decards. Hence,
$\adern(D_{m,m+2,p})\geq 3$. If $|\cS|=3$ and $m\geq 3$, by
Lemmas~\ref{leafhub},~\ref{leafmiddle}, and~\ref{hubmiddle}, we
may assume the decards in $\cS$ are all leaf decards or middle
decards. By Lemmas~\ref{middle} and~\ref{oneleaf}, the
reconstruction $G$ is $D_{m,m+2,p}$. For $D_{1,3,p} (p\neq 3)$,
by the proofs of
Lemmas~\ref{middle},~\ref{oneleaf},~\ref{leafhub},~\ref{leafmiddle},
and~\ref{hubmiddle}, we may assume assume $\cS$ has $2$ middle
decards and the left leaf decard. The reconstruction from the
left leaf decard other than $D_{1,3,p}$ is $D_{1,3,p-1}^{0,1}$.
But $D_{1,3,p-1}^{0,1}$ shares two decards with $D_{1,3,p}$. We
have $\adern(D_{m,m+2,p})\leq 3 (m\notin \{2,3\} , p\neq 3)$ and
$\adern(D_{3,5,2})\leq 3$.

{\it Case 5: Otherwise, some edge fails the condition
of~\ref{leme1}.} Since some edge fails the condition
of~\ref{leme1}, $\adern(D_{m,n,p})\geq 2$. By the proofs of
Lemmas~\ref{hub},~\ref{middle},~\ref{oneleaf},~\ref{leafhub},~\ref{leafmiddle},
and~\ref{hubmiddle}, any two decards determine $D_{m,n,p}$. We
have $\adern(D_{m,n,p})=2$.
\end{proof}

Since the double-star is $D_{m,n,2}$, we can deduce the conclusion
which is obtained in~\cite{mse}.

\begin{corollary} For the double-star $D_{m,n,2}$ with $1 \leq m
\leq n$, $\dern(D_{m,n,2})=1$ and
$$\adern(D_{m,n,2})=
 \begin{cases}
    3,   & \textrm{either $n=m+2$ or $2\in\{m,n\}$;}\\
    1,   & \textrm{otherwise.}
\end{cases}$$
\end{corollary}

\section*{Acknowledgements}
This work was supported by the National Natural Science
Foundation of China (Grant No. 11101378) and Zhejiang Natural
Science Foundation (Grant No. LY14A010009).

\end{document}